\documentclass[11pt, oneside]{amsart}
\usepackage[english]{babel}
\usepackage{hyperref}
\usepackage{cleveref}
\usepackage{amssymb, amsfonts, amsmath}
\usepackage{comment} 
\usepackage{mathtools}
\usepackage{todonotes}
\usepackage{color}
\usepackage{wrapfig}
\usepackage{tikz}
\usepackage{tikz-cd}
\usepackage{comment}
\usepackage{mathrsfs}
\usepackage{tabularx, hyperref}
\usepackage{amssymb} \usepackage{amsfonts} \usepackage{amsmath}
\usepackage{amsthm} \usepackage{epsfig, subfig}
\usepackage{ amscd, amsxtra, latexsym}
\usepackage{epsfig,  graphicx, psfrag}
\usepackage[all]{xy}
\usepackage{caption}
\usepackage{enumerate}
\usepackage{color}
\DeclareMathOperator{\vol}{Vol}

\newcommand{\str}{{{\mathrm{str}}}}

\newcommand{\R} {\ensuremath {\mathbb{R}}}

\newcommand{\norm}[1]{{\left\|#1\right\|}}

\newcommand*{\ucM}{{\widetilde{M}}}

\newcommand\restr[2]{{
\left.\kern-\nulldelimiterspace 
#1 
\vphantom{\big|} 
\right|_{#2} 
}}

\newcommand\rrestr[2]{{
\left.\kern-\nulldelimiterspace 
\left.\kern-\nulldelimiterspace 
#1 
\vphantom{\big|} 
\right|\hspace{-2.4pt} 
\right|_{#2} 
}}

\renewcommand{\phi}{\varphi}
\renewcommand{\colon }{\,:}

\newcommand*{\Simpl}{\mathrm{Simpl}}
\newcommand*{\STR}{\mathrm{str}}

\newtheorem{lemma}{Lemma}[section]
\newtheorem{thm}[lemma]{Theorem}
\newtheorem{prop}[lemma]{Proposition}

\newtheorem*{prop*}{Proposition}
\newtheorem{prop_intro}{Proposition}

\newtheorem{quest_intro}[prop_intro]{Question}
\newtheorem{thm_intro}[prop_intro]{Theorem}

\newtheorem{cor_intro}[prop_intro]{Corollary}

\theoremstyle{definition}

\newtheorem{defn}[lemma]{Definition}

\theoremstyle{remark}
\newtheorem{rmk}[lemma]{Remark}

\newtheoremstyle{citing}
{3pt}
{3pt}
{\itshape}
{}
{\bfseries}
{.}
{.5em}
{\thmnote{#3}}
\theoremstyle{citing}

\begin{document}

\title{On the cup product of De Rham classes in bounded cohomology}

\author[Roberto Frigerio]{Roberto Frigerio}
\address{Dipartimento di Matematica, Universit\`a di Pisa, Italy}
\email{roberto.frigerio@unipi.it}

\author[Francesco Milizia]{Francesco Milizia}
\address{Dipartimento di Matematica, Universit\`a di Bologna, Italy}
\email{francesco.milizia@unibo.it}



\begin{abstract}
On a negatively curved closed manifold, there exists a well-defined map $\Psi^\bullet$
associating to every closed differential form a bounded 
cohomology class via integration over straight simplices. Classes in the image of this map, which, a priori, depend on the fixed family of straight simplices, are usually called \emph{De Rham} classes, and constitute an interesting subspace of bounded cohomology.

In this paper we prove that, in sufficiently high degrees, $\Psi^\bullet$ is a homomorphism of algebras, i.e., it sends the wedge product of closed differential forms to the cup product of the associated bounded cohomology classes.
The degree in which $\Psi^\bullet$ starts to preserve products depends on 
 the  boundedness of Jacobians of straight simplices. For the barycentric straightening introduced by Besson, Courtois and Gallot, this happens for degrees $\ge 3$.

As a corollary, the cup product of two De Rham classes  vanishes, provided that its degree exceeds the dimension of the manifold
(and the degrees of both classes are $\geq 3$).
This result complements vanishing results for the cup product of De Rham classes due to Marasco and to Battista et al.
\end{abstract}

\maketitle

\section{Introduction}

Let $M$ be a closed negatively curved Riemannian $n$-manifold, $n\geq 2$.
It is known that
 the bounded cohomology $H^k_b(M)$ of $M$ with real coefficients is infinite dimensional in degrees $2$ and $3$, 
and
every standard cohomology class with real coefficients in degree at least 2 admits a bounded representative. 
However, the bounded cohomology module $H^k_b(M)$
may  still be quite mysterious when $k>n$ and $k>3$: for example, for no closed hyperbolic $3$-manifold $M$  is it known whether $H^k_b(M)$ vanishes or not for any $k>3$.

A natural question is whether non-trivial bounded cohomology classes in degree $k\geq 4$ may be constructed by taking the cup product of lower degree classes. In this paper
we provide  vanishing results for the cup product of bounded cohomology classes arising from integration of differential forms, complementing some theorems proved 
in~\cite{Marasco},~\cite{Battista} and~\cite{BuFri}, where the cup product of \emph{degree--2} differential forms with generic bounded cohomology classes is considered.

Thanks to the negative curvature of $M$, it is possible to define a \emph{straightening} of singular simplices in $M$. 
Several forms of straightening exist: we refer the reader to Subsections~\ref{str1:sub} and~\ref{str2:sub} for a detailed discussion of this notion and of geometric properties of straightenings. 
For reasons that will be thoroughly discussed below, in our Theorem~\ref{main:thm} and Corollary~\ref{main:cor}
we consider the \emph{barycentric} straightening originally due to Besson, Courtois and Gallot~\cite{BCG0,BCG} and studied by a number of authors~\cite{LFS,LFW,connell_positivity_2019,CW2,CW}, rather than the geodesic straightening
described, e.g., in~\cite{inoue_gromov_1982}.

Let $\Omega^k(M)$ be the space of degree--$k$ differential forms on $M$, and denote by $C\Omega^k(M)$ the subspace of closed forms. 
For every $k\in\mathbb{N}$, every singular simplex $\sigma\colon \Delta^k\to M$ has an associated (baricentrically) \emph{straightened} simplex $\str_k(\sigma)\colon \Delta^k\to M$.
Moreover, under the assumption $k\geq 3$,  the $k$-dimensional volume of $\str_k(\sigma)$ is uniformly bounded above by a constant only depending on the curvature of $M$.
Hence,
integration over straight simplices yields a well-defined chain map
$$
\Psi^{k} \colon C\Omega^{k}(M)\to H^{k}_b (M)
$$
such that, for every $\omega\in C\Omega^k(M)$, $k\geq 3$, a representative of the class $\Psi^k(\omega)$ is given by the bounded singular cochain such that
$$
\sigma \mapsto \int_{\str_k(\sigma)} \omega
$$
for every singular simplex $\sigma\colon \Delta^k\to M$. 
(An even stronger result holds true when working with the geodesic straightening: in that case, one needs to assume only $k\geq 2$ rather than $k\geq 3$.)
Classes lying in the image of $\Psi^{\bullet}$ are called \emph{De Rham classes} (note that this notion depends in general on the fixed straightening on $M$). 

Barge and Ghys proved in~\cite{bargeghys} the rather surprising result that, when $M$ is a closed negatively curved surface and $k=2$, 
integration along (geodesically) straight simplices provides an \emph{injective} map
$C\Omega^2(M)\to H^2_b(M)$; in particular, the second bounded cohomology of such surfaces contains
an infinite dimensional space, whose classes are explicitly described by
integration of forms. This result  has been recently generalized in~\cite{Battista} to the case when $M$ is a closed negatively curved Riemannian manifold 
of an arbitrary dimension, and in~\cite{DAFH} to many open hyperbolic manifolds. Thus, at least in degree 2, closed differential forms may be exploited to define a very wide family of interesting bounded 
cohomology classes.

We prove here the following:

\begin{thm_intro}\label{main:thm}
Let $M$ be a  closed negatively curved Riemannian manifold. Then the integration map along baricentrically straightened simplices 
$$
\Psi^\bullet \colon C\Omega^\bullet(M)\to H^\bullet_b (M)
$$
is a homomorphism of algebras in sufficiently high degrees, i.e.,
$$
\Psi^{p+q}(\omega_1\wedge\omega_2)=\Psi^p(\omega_1)\cup \Psi^q(\omega_2)
$$
for every $p\geq 3$, $q\geq 3$, and every $\omega_1\in C\Omega^p(M)$,  $\omega_2\in C\Omega^q(M)$.
\end{thm_intro}

\begin{cor_intro}\label{main:cor}
Let $M$ be a closed $n$-dimensional negatively curved manifold, and let $\beta_1\in H^p_b(M)$, $\beta_2\in H^q_b(M)$ be De Rham classes associated
to the barycentric straightening of $M$. If $p\geq 3$, $q\geq 3$ and $p+q>n$, then 
\[
\beta_1\cup \beta_2=0 \ \in \ H^{p+q}_b(M)\ .
\]
\end{cor_intro}
\begin{proof}
By definition of De Rham class, we have $\beta_1=\Psi^p(\omega_1)$, $\beta_2=\Psi^q(\omega_2)$ for some
$\omega_1\in C\Omega^p(M)$, $\omega_2\in C\Omega^q(M)$. Since $p+q>n$, we have trivially
$\omega_1\wedge \omega_2=0$, hence $\beta_1\cup \beta_2=0$ by Theorem~\ref{main:thm}.
\end{proof}	

There exist by now quite a number of results on the vanishing of the cup product in bounded cohomology. 
Marasco proved in~\cite{Marasco} that, for a closed negatively curved manifold of any dimension, the cup product of any degree--2 \emph{exact} De Rham class (defined via the geodesic straightening) with any other bounded cohomology class vanishes. Very recently, Bucher and the first author extended this result to the cup product
of the bounded volume form with any other bounded cohomology class~\cite{BuFri}. These results imply that, for closed negatively curved surfaces,
De Rham classes defined via the geodesic straightening belong to the annihilator of the cup product. 
Probably driven by the long-standing problem whether the bounded cohomology of free groups vanishes in degree $\geq 4$, a lot of attention has been devoted to the case of the free
group, for which it has been shown that cup products of bounded cohomology classes vanish in a wide variety of cases
(see~\cite{Heuer, BucherMonod,Bucher2}).

\subsection{Straightenings in negative curvature}\label{str1:sub}

We believe that Theorem~\ref{main:thm} (and then Corollary~\ref{main:cor}) should be true without the restriction $p,q \ge 3$ on the degrees of the closed forms involved. Surprisingly enough,
what prevents us from proving this more general result is the lack of boundedness properties of geometric straightenings in negative curvature. 
Roughly speaking, a straightening is a chain map (chain homotopic to the identity) $C_\bullet(M)\to C_\bullet(M)$ associating to every simplex $\sigma$ with values in $M$ a \emph{straight} simplex with better regularity properties (see Subsection~\ref{str2:sub} for a precise definition). When exploiting straightenings in bounded cohomology (or for the study of simplicial volumes, as for example in Gromov--Thurston's computation of the simplicial volume of hyperbolic manifolds~\cite{thurston_geometry_1979}), a key fact is that the volume of straight simplices (of dimension at least 2) should be uniformly bounded:  when this is the case, we say that the straightening \emph{has bounded volumes} in degree $\geq 2$ (see Subsection~\ref{boundedness:sub}).
Such a straightening allows to define the map $\Psi^{k} \colon C\Omega^{k}(M)\to H^{k}_b (M)$ for every $k\geq 2$. However, 
in order to prove  Theorem~\ref{main:thm}, we need a stronger property: namely, we need a control on the norm of the differential of any straight simplex at any of its point 
(see again Subsection~\ref{boundedness:sub} for the precise definition of \emph{having bounded Jacobians} in degree $\geq d$).

As already mentioned, there are (at least) two straightenings which have been repeatedly considered 
in the literature for negatively curved manifolds: the geodesic straightening and the more sophisticated
barycentric straightening due to Besson, Courtois and Gallot. In the case of constant negative curvature, i.e., for hyperbolic manifolds, the geodesic straightening and
the barycentric straightening coincide up to reparametrization, hence they define the same map $\Psi^k\colon C\Omega^k(M)\to H^k_b(M)$, $k\geq 2$. In the general case, however, they could in principle give rise to different maps $\Psi^k$, $k\geq 3$ (see Questions~\ref{bounded2:quest} and~\ref{distint:quest} below).

 It is well known that the geodesic straightening has bounded volumes in degree $\geq 2$~\cite{inoue_gromov_1982}; however, even in the hyperbolic case, it has not
bounded Jacobians in any degree (see Subsection~\ref{geodesic:subsec}). On the contrary, it is known that the barycentric straightening has bounded Jacobians
in top degree $n=\dim M$, when $n\geq 3$. In Subsection~\ref{bounded:jac:sub} we extend this result to all degrees $\geq 3$, by proving the following result
(probably known to experts): 

\begin{thm_intro}\label{bounded:jac:thm}
Let $M$ be a closed negatively curved Riemannian manifold. Then, the barycentric straightening on $M$ has bounded Jacobians in degrees $\geq 3$.
\end{thm_intro}

Theorem~\ref{bounded:jac:thm} plays an important role in our proof of Theorem~\ref{main:thm}.
 We also prove in Subsection~\ref{unbounded:sub} that the barycentric straightening has \emph{not} bounded Jacobians in degree 2, even in the constant curvature case:
 
 \begin{thm_intro}\label{nonbounded:jac:thm}
 The barycentric straightening on the hyperbolic plane $\mathbb{H}^2$ does not have bounded Jacobians in degree 2.
 \end{thm_intro}

Observe that, while the property of having bounded volumes is not sensitive to reparametrizations, the property of having bounded Jacobians strictly depends on the parametrization
of straight simplices. Therefore, it makes sense to ask the following:

\begin{quest_intro}
Is it possible to modify the geodesic straightening (or the barycentric straightening) by taking new parametrizations of its straight simplices in order to obtain a new straightening with bounded Jacobians
in degrees $\geq 2$?
\end{quest_intro} 

More in general:

\begin{quest_intro}
Let $M$ be a closed negatively curved manifold. Does there exist a straightening on $M$ having bounded Jacobians in degrees $\ge 2$?
\end{quest_intro} 

A positive answer to any of these questions would allow us (aside from some technicalities concerning the compatibility of the straightening with the cross product, see Subsection \ref{cross:sec})
to prove Theorem~\ref{main:thm} without any restriction on the degrees of the differential forms involved
(recall that $H^1_b(M)=0$, hence it only makes sense to consider forms of degree $\geq 2$). 

As mentioned above, in the hyperbolic case the barycentric straightening coincides, up to reparametrization, with the geodesic one, hence it has bounded volumes in degree $\geq 2$. It is not clear, however, whether this fact also holds in variable curvature:

\begin{quest_intro}\label{bounded2:quest}
Has the barycentric straightening bounded volumes also in degree 2?
\end{quest_intro}

More in general, one may wonder whether the geodesic and the barycentric straightening associate to a closed differential form the same bounded cohomology class:

\begin{quest_intro}\label{distint:quest} 
Let $\Psi^k\colon C\Omega^k(M)\to H^k_b(M)$ be the map of our Theorem~\ref{main:thm}, defined via
the barycentric straightening.
Let $\Psi^k_{\text{gd}}\colon C\Omega^k(M)\to H^k_b(M)$  be the analogous map obtained
via the geodesic straightening. Are $\Psi^k$ and
$\Psi_{\text{gd}}^k$ equal for every $k\geq 3$ (or, should Question~\ref{bounded2:quest} have a positive answer, for every $k\geq 2$)?
\end{quest_intro}

The question of the dependence of the integration map $\Psi^k$ on the chosen straightening is closely related to \cite[Question 8]{bbcm_de_rham}, where the authors ask to what extent the map depends on the fixed \emph{metric} on the manifold. In fact, when different metrics are considered, one obtains different geodesic (or barycentric) straightenings, and it is not clear how this affects the integration map $\Psi^k$ and its image in bounded cohomology.

\subsection*{Acknowledgements} The authors thank Chris Connell, Stefano Francaviglia, Jean--Fran\c{c}ois Lafont and Shi Wang for useful conversations. The authors are partially supported by INdAM through GNSAGA.
Francesco Milizia is supported by the ERC ``Definable Algebraic Topology'' DAT - Grant Agreement no. 101077154.

\section{Preliminaries}
Let $M$ be a topological space (in fact, we will be interested only in the case when $M$ is a manifold).
For every $k \in \mathbb{N}$, we denote by $\Delta^k$ the standard $k$-simplex, which is the convex hull of the canonical basis of $\mathbb{R}^{k+1}$, and by
$\Simpl_k(M)$ the set of singular $k$-dimensional simplices of $M$, i.e., the set whose elements are continuous maps from $\Delta^k$ to $M$.
We denote by $C_\bullet (M)$ the singular chain complex of $M$ with real coefficients, which is given, in degree $k$, by the real vector space
with basis $\Simpl_k(M)$.
Henceforth, we implicitly understand that (co)chains and (co)homology groups are always with real coefficients.

\subsection{Bounded cohomology}

The \emph{bounded cohomology} $H^\bullet_b(M)$ of $M$ is the cohomology of the complex
\[
	0\rightarrow C^0_b(M)\rightarrow C^1_b(M)\rightarrow C^2_b(M)\rightarrow\cdots \,,
\]
where $C^n_{b}(M)$ denotes the space of bounded singular $n$-cochains of $M$, the differential maps are the restrictions of the boundary maps of the usual singular cochain complex $C^\bullet(M)$ and a cochain $\varphi \in C^n(M)$ is \emph{bounded} if
\[
	\lVert \varphi \rVert_\infty =
	\sup \bigl\{ | c(\sigma) |, \; \sigma \in \Simpl_n(M)
	\bigr\} < \infty .
\]

\subsection{Straightenings}\label{str2:sub}

We first provide an abstract definition of straightening for simplices which, in particular, will encompass both the geodesic and the barycentric straightening.

\begin{defn}\label{defn:straightening}
	Let $M$ be a smooth manifold.
	A straightening $\STR_\bullet$ on $M$ is a sequence of maps
	\[ \STR_k \colon \Simpl_k(M) \to \Simpl_k(M), \]
	where $k$ varies in $\mathbb{N}$, with the following properties:
	\begin{enumerate}
		\item For every $k$-simplex $\sigma$, its straightening $\STR_k(\sigma)\colon \Delta^k \to M$ is a function of class $C^1$;
		\item The $i$-th face of the straightening of $\sigma \in \Simpl_k(M)$ is equal to the straightening of the $i$-th face of $\sigma$, for every $i \in \{0,\dots,k\}$;
		\item There is a family of homotopies $h_\sigma \colon \Delta^k \times [0,1] \to M$ connecting $\sigma$ and $\STR_k(\sigma)$, where $k$ varies in $\mathbb{N}$ and $\sigma$ varies in $\Simpl_k(M)$; this family of homotopies is compatible with restriction to faces, i.e., the homotopy associated to a  face of $\sigma$ 
		is equal to the restriction of $h_\sigma$ to the same face.
	\end{enumerate}
\end{defn}

The linear extension of a straightening to $C_\bullet(M)$ provides  a chain map $\STR_\bullet\colon C_\bullet(M) \to C_\bullet(M)$, which is homotopic to the identity.

We will consider, in particular, straightenings on manifolds which have negative sectional curvature, and are defined using geometric methods; this explains why we speak of ``straightening'' instead of ``regularization'', even though geometry does not appear in \Cref{defn:straightening}.

\begin{rmk}[Comparison with definitions from other papers]
	Our definition of straightening coincides with the one considered by Lafont-Wang in \cite{LFW} (where, however, the authors focus on manifolds which are symmetric spaces of noncompact type).
	The definition given by Lafont-Schmidt and Connell-Wang in \cite{LFS,CW} is only slightly different. 	Instead, Gromov, in \cite{Gromov82}, assumes that the straightening of a simplex $\sigma$ should depend only on the homotopy classes (relative to the endpoints) of the edges of $\sigma$ --- in other words, only on the vertices of a lift of $\sigma$ on the universal cover of the manifold; this is in fact true for both the geodesic and the barycentric straightenings.
\end{rmk}

\begin{rmk}
	Our definition of straightening does \emph{not} require $\str_0$ to be the identity. Thus, in general, a straightened simplex could have vertices which are distinct from the ones of the original simplex.
	The geodesic straightening is special in that it fixes vertices (this has some useful consequences: for example, in that case the straightening is idempotent, i.e., it doesn't affect straight simplices); for the barycentric straightening, this is not  true in general (but still holds in specific cases, e.g., for hyperbolic manifolds). 
\end{rmk}

	It readily follows from the properties listed in \Cref{defn:straightening} that, for any connected manifold $M$,  straightenings on $M$ 
	correspond to	straightenings on the universal covering of $M$ that are equivariant (and equivariantly homotopic to the identity, as described in condition (3) above) 
	with respect to deck transformations.

\subsection{Boundedness properties of straightenings}\label{boundedness:sub}

Let $M$ be a Riemannian manifold and $\STR_\bullet$ be a straightening on $M$.
\begin{defn}\label{defn:bounded_areas}
	Let $d \in \mathbb{N}$.
	We say that $\STR_\bullet$ has \emph{bounded volumes} in degree $d$ if there exists a constant $C_d \in [0,+\infty)$ such that for any $\sigma \in \Simpl_d(M)$ its straightening $\STR_d(\sigma)$ has $d$-volume not greater than $C_d$.
\end{defn}

For every $p\in\Delta^k$, the \emph{norm of the $d$-Jacobian} of a $C^1$ singular simplex $\sigma\colon \Delta^k\to M$ is the supremum of the absolute values
of $\|d\sigma_p (v_1)\wedge\ldots \wedge d\sigma_p(v_k)\|$, as $(v_1,\ldots,v_k)$ varies in the set of orthonormal $k$-frames
in the tangent spaces $T_p \Delta^k$, $p\in\Delta^k$, and the norm $\|\cdot \|$ on $\Lambda^k (T_{\sigma(p)} M)$ is induced by the Riemannian metric of $M$.

\begin{defn}
	We say that $\STR_\bullet$ has \emph{bounded Jacobians} in degree $d$ if there exists a constant $L_d \in [0,+\infty)$ such that, for every $k \ge d$ and every $\sigma \in \Simpl_k(M)$, the norm of the $d$-Jacobian of $\STR_k(\sigma)$ is at most $L_d$.  
\end{defn}

If $\STR_\bullet$ has bounded Jacobians in degree $d$, then it has bounded volumes in degree $d$, since the $d$-volume of a $d$-dimensional simplex is obtained by integrating its $d$-Jacobian on the standard simplex $\Delta^d$.


\subsection{Integration along straight simplices}
	Let $M$ be a smooth manifold and $\STR_\bullet$ be a straightening on $M$.
	We define the integration map
	\[ I_{\STR}^\bullet \colon \Omega^\bullet(M) \to C^\bullet(M) \]
	by setting, for every $\omega \in \Omega^k(M)$ and $\sigma \in \Simpl_k(M)$,
	\[ I_{\STR}^k(\omega)(\sigma) = \int_{\STR_k(\sigma)} \omega\, ,
	\]
	and extending linearly $I_{\STR}^k(\omega)$ on $C_k(M)$ for every $k\in\mathbb{N}$.
Stokes' Theorem (together with the fact that the straightening commutes with taking faces) readily implies that $I_{\STR}^\bullet$ is a chain map. Moreover, it is easily seen that $I_{\STR}^\bullet$
induces the classical De Rham isomorphism between De Rham cohomology and standard singular cohomology.

Let us now suppose that $M$ be compact. Then, every differential form on $M$ is bounded. Hence,  if $\STR_\bullet$ has bounded volumes in degree $k$, then $I_{\STR}^k$ factors through $C_b^k(M)$, thus defining a map $\Omega^k(M) \to C_b^k(M)$, which will still be denoted by
$I^{\STR}_k$. 

\begin{defn}\label{def:psi}
	Let $\STR_\bullet$ be a straightening on $M$ with bounded volumes in degree $k$.
	Let $C\Omega^k(M)$ be the vector space of closed $k$-forms on $M$.
	We denote by
	\[ \Psi^{k}_\STR \colon C\Omega^k(M) \to H_b^k(M) \]
	the map defined as $\Psi_{\STR}^k(\omega) = [I_{\STR}^k(\omega)]$, for every $\omega \in C\Omega^k(M)$.
\end{defn}

\section{Geometric straightenings}
In this section we discuss the boundedness properties of the \emph{geodesic} and of the \emph{barycentric} straightenings. Throughout 
the section, we denote by $M$ a closed negatively curved Riemannian manifold with Riemannian universal covering
$\widetilde{M}\to M$. We also denote by $\Gamma<\text{Isom}(\widetilde{M})$  the 
group of the covering automorphisms of $\widetilde{M}$, so that $M=\widetilde{M}/\Gamma$.

\subsection{The geodesic straightening}\label{geodesic:subsec}

Since $M$ is compact, hence complete, the universal covering $\widetilde{M}$
is itself complete. Being also negatively curved,  it is uniquely geodesic, and this allows us to define a straightening on $\widetilde{M}$
as follows.

Let $k\in\mathbb{N}$ and let $x_1,\dots,x_{k+1}$ be points in $\widetilde{M}$. The \emph{straight} simplex $[x_1,\dots,x_{k+1}]\in  \Simpl_k(\widetilde{M})$
with vertices $x_1,\dots,x_{k+1}$ is defined as follows: if $k=0$, then $[x_1]$ is the constant $0$-simplex with image $x_1$; if straight simplices
have been defined for every dimension $h < k$, then $[x_1,\dots,x_{k+1}]\in  \Simpl_{k}(\widetilde{M})$ is determined by the following condition: 
for every $z\in \Delta^{k-1}\subseteq \Delta^{k}$ (here, $\Delta^{k-1}$ is identified with the face of $\Delta^{k}$ opposite to the last vertex $e_{k+1}$ of $\Delta^{k}$), the restriction of $[x_1,\dots,x_{k+1}]$ to the segment with endpoints
$z,e_{k+1}$ is the constant-speed parametrization of the geodesic joining $[x_1,\dots,x_k](z)$ to $x_{k+1}$. The fact that $[x_1,\dots,x_{k+1}]$
is indeed of class $C^1$ is proved in~\cite[Proposition 2.4]{loehsauer}.

If $\sigma\colon \Delta^k\to M$ is any singular simplex, we then define $\str^{\rm gd}_k(\sigma)$ as follows: we lift $\sigma$ to $\widetilde{\sigma}\colon \Delta^k\to \widetilde{M}$,
we set $x_i=\widetilde{\sigma}(e_i)$ for $i=1,\dots, k+1$, and we finally set $\str^{\rm gd}_k(\sigma)=\pi\circ [x_1,\dots,x_{k+1}]$, where $\pi\colon \widetilde{M}\to M$ is the covering projection. 
 
It is proved in~\cite{inoue_gromov_1982} that the geodesic straightening has bounded volumes in degree $d\geq 2$. Unfortunately, it has not bounded Jacobians in any degree. Indeed, if $\dim M=n$ and $x_{k+1}$ is any point in $\widetilde{M}$, then for every $1\leq k\leq n$ and every $L>0$ we may choose an orthonormal $k$-frame $(v_1,\dots,v_{k})$ in the tangent space $T_{x_{k+1}} \widetilde{M}$ and set $x_i=\gamma_i(L)$, where $\gamma_i$ the geodesic of $\widetilde{M}$ starting at $x_{k+1}$ with initial speed $v_i$. 
Then, the norm of the $k$-Jacobian of the straight simplex $[x_1,\dots, x_{k+1}]$ at $e_{k+1}$ is equal to $L^k$. Due to the arbitrariness of $L$, this implies that the geodesic straightening
has not bounded Jacobians. This is the main reason why in this paper we mainly consider the barycentric straightening,
which was first introduced by Besson, Courtois and Gallot.

\subsection{The barycentric straightening}\label{bary:sub}

Let $\partial_\infty \ucM$ be the boundary at infinity of $\ucM$, defined, e.g., as the set of geodesic rays in $\ucM$ modulo the equivalence relation of travelling at bounded distance.
Recall that $\partial_\infty \ucM$ has a topology (inherited, e.g., from the unit tangent sphere at any point of $\ucM$), which makes it homeomorphic to a sphere. In order to define the barycentric straightening, we first need to introduce 
the \emph{Patterson-Sullivan measures} of $\ucM$, which are a family of Radon probability measures $\{\nu_x\}_{x \in \ucM}$ on $\partial_\infty \ucM$, parametrized by points of $\ucM$, satisfying the following properties:
\begin{enumerate}
	\item \label{ps-invariance} $\gamma_* \nu_x = \nu_{\gamma x}$ for every $\gamma \in \Gamma$ and $x \in \ucM$;
	\item The measures $\nu_x$ are all absolutely continuous with respect to each other, and their Radon-Nikodym derivatives are \[\frac{d\nu_y}{d\nu_x}(\xi) = \exp(-h_\ucM \cdot b_{\xi,x}(y)),\]
	for every $x,y \in \ucM$.
\end{enumerate}
In the expression above, $h_\ucM$ is the volume entropy of $\ucM$, which, since $\widetilde{M}$ has negative curvature bounded away from zero, is a positive real number; moreover, $b_{\xi,x}\colon \ucM \to \mathbb{R}$ denotes the Busemann function on $\ucM$ relative to the point at infinity $\xi$ (its gradient, at any point, has unit length and points to the direction opposite to $\xi$) and attaining the value $0$ at $x$. For the existence (and the uniqueness) of 
the Patterson-Sullivan measures $\{\nu_x\}_{x \in \ucM}$ see, e.g., \cite{knieper_asy}.

Conversely, given a measure $\nu$ on $\partial_\infty \ucM$, its barycenter in $\ucM$ is obtained by considering the function
\[ \mathcal{B}_{\nu,x_0} (x) = \int_{\partial_\infty \ucM} b_{\xi, x_0}(x) \ d \nu(\xi) \]
and defining $\mathrm{Bar}(\nu) \in \ucM$ as the point where  $\mathcal{B}_{\nu,x_0}$ attains its minimum.
Here, $x_0 \in \ucM$ is an uninfluential basepoint, whose choice only changes $\mathcal{B}_{\nu,x_0}$ by adding a constant.
Since $M$ is compact, the Patterson-Sullivan measures $\nu_x$ have full support \cite[Lemma 4.1]{knieper_asy}, 
and this implies in turn that 
the function $\mathcal{B}_{\nu,x_0}$ is strictly convex whenever $\nu$ is a positive linear combination of Patterson-Sullivan measures
(this follows from \Cref{prop:Hess} below, see \cite[Lemma 3.1]{connell_positivity_2019}).
Therefore, in this case the barycenter $\mathrm{Bar}(\nu)$ is well defined.

\begin{defn}[cf. \cite{connell_positivity_2019}]
	The barycentric straightening on $M$ is defined as follows.
	Given $\sigma \in \Simpl_k(M)$, take a lift of $\sigma$ to the universal covering $\ucM$, and let $x_1, \dots, x_{k+1} \in \ucM$ be its vertices.
	Then, $\STR_k(\sigma)$ is defined by sending any $a =(a_1, \dots, a_{k+1}) \in \Delta^k$ to the projection on $M$ of the point
	\[ \mathrm{Bar}(a_1^2 \nu_{x_1} + \dots + a_{k+1}^2 \nu_{x_{k+1}}) \in \ucM, \]
	where the $\nu_{x_i}$ are the Patterson-Sullivan probability measures associated to the vertices of the lifted simplex.
	
	Since the construction performed on $\ucM$ is equivariant with respect to deck transformations, the result does not depend on the chosen lift of $\sigma$, and the procedure is a straightening according to \Cref{defn:straightening}; 
	the $C^1$ regularity of $\STR_k(\sigma)$ follows from the implicit function theorem applied to equation \eqref{eq:stationary} below, using the fact that Busemann functions are $C^2$ and have good convexity properties (see \Cref{prop:Hess}).
\end{defn}

\subsection{The barycentric straightening has bounded Jacobians}\label{bounded:jac:sub}

This subsection is devoted to the proof of Theorem~\ref{bounded:jac:thm}, which states
 that the barycentric straightening has bounded Jacobians in degrees $\ge 3$.
Since our argument follows very closely the one in \cite{connell_positivity_2019}, we skip some details, but aim at carefully describing the strategy and the key points 
of the proof.

Let $k \in \mathbb{N}$, fix $k+1$ points on $\ucM$ (which should be thought as the vertices of a lifted singular simplex), and call them $x_1, \dots, x_{k+1}$.
Let $\nu_1, \dots, \nu_{k+1}$ be the associated Patterson-Sullivan probability measures.
Fix also an (uninfluential) basepoint $x_0 \in \ucM$.

By definition, the barycenter method produces a map $s\colon\Delta^k\to \ucM$, the ``straightened simplex'', which sends any $a = (a_1, \dots, a_{k+1}) \in \Delta^k \subset \mathbb{R}^{k+1}$ to the (unique) point $x \in \ucM$ that minimizes the quantity
\[ \int_{\partial_\infty \ucM} b_{\xi,x_0}(x) \sum_{i=1}^{k+1} a_i^2 d\nu_i(\xi). \]
By differentiating this quantity with respect to $x$ and evaluating at $s(a)$, one obtains, for every $a \in \Delta^k$:
\begin{equation}\label{eq:stationary}
	0 \ =\  \int_{\partial_\infty \ucM} db_{\xi,x_0}(s(a)) \sum_{i=1}^{k+1} a_i^2 d\nu_i(\xi)\ \in\ (T_{s(a)}\ucM)^*.
\end{equation}
Differentiating now with respect to $a$, one then obtains that, for every $a \in \Delta^k$, $u = (u_1,\dots,u_k) \in T_a\Delta^k$ and $v \in T_{s(a)}\ucM$,
\begin{align}\label{eq:derived_relation}
	0 &=\ \sum_{i=1}^{k+1} 2a_iu_i\int_{\partial_\infty \ucM} db_{\xi,x_0}(v) d\nu_i(\xi) \\
	&\quad + \int_{\partial_\infty \ucM} \mathrm{Hess}\ b_{\xi,x_0}(ds(u), v) \sum_{i=1}^{k+1} a_i^2d\nu_i(\xi). \nonumber
\end{align}

Let $V$ be a $d$-dimensional linear subspace of $T_a\Delta^k$, with $3 \le d \le k$.
We wish to prove that there is an upper bound, which depends only on the geometry of $\ucM$, for the (absolute value of) the determinant of $ds_{\restriction V} \colon V \to ds(V)$.
Of course, we may suppose that the tangent map $ds$ is injective on $V$ (otherwise, there is nothing to prove).
Consider the following two symmetric positive semidefinite bilinear forms on $ds(V)$: 
\begin{align*}
	K_a (w,v) &=  \int_{\partial_\infty \ucM} \mathrm{Hess}\ b_{\xi,x_0}(w, v) \sum_{i=1}^{k+1} a_i^2d\nu_i(\xi),\\
	H_a(w,v) &= \int_{\partial_\infty \ucM} db_{\xi,x_0}(w) \cdot db_{\xi,x_0}(v) \sum_{i=1}^{k+1} a_i^2d\nu_i(\xi).
\end{align*}

It follows from a manipulation of \eqref{eq:derived_relation}, see \cite[p.~1016]{connell_positivity_2019} for details, that
\begin{equation}\label{eq:KH}
	\lvert K_a(ds(u), v) \rvert \le 2 \norm{u} \cdot \sqrt{H_a(v,v)}
\end{equation}
for every $a \in \Delta^k$, $u \in T_a\Delta^k$ and $v \in ds(T_a\Delta^k)$. 
Moreover, $K_a$ is positive definite (because of \Cref{prop:Hess} below).
A crucial consequence of \eqref{eq:KH} is the following estimate:
\begin{equation}\label{eq:J_estimate}
	\lvert \det (ds_{\restriction V}) \rvert \le 2^d\ \frac{\sqrt{\det H_a}}{\det K_a}.
\end{equation}
The strategy is now to show that $K_a$ has at most \emph{one} ``small'' eigenvalue, and that $H_a$ has at least \emph{two} comparatively small eigenvalues (while all its eigenvalues are $\le 1$, as it follows from its definition and the basic properties of Busemann functions), so that the the right-hand side of~\eqref{eq:J_estimate} can be controlled.
Actually, the proof shows that $H_a$ has $d-1$ small eigenvalues, so that the conclusion follows when $d \ge 3$, as claimed.

The main ingredients of the proof are the following two propositions, which are simplified versions of statements in \cite{connell_positivity_2019}, adapted to our situation in which  $M$ (and, hence, $\ucM$) have negative curvature bounded away from zero.

\begin{prop}[{\cite[Theorem 5]{connell_positivity_2019}}]\label{prop:Hess}
	Let $\xi \in \partial_\infty \ucM$ and $x \in \ucM$.
	Then, for every $v \in T_x\ucM$ 
	orthogonal to the direction pointing towards $\xi$ (or, equivalently, tangent to the horosphere centered at $\xi$ and passing through $x$), it holds
	\[ \mathrm{Hess}\ b_{\xi,x_0}(v,v) \ge C \cdot \lVert v \rVert^2,\]
	for some $C > 0$ independent of $x,v$ and $\xi$.
\end{prop}

\begin{prop}[{\cite[Lemma 5.3]{connell_positivity_2019}}]\label{prop:compare_HK}
	Let $\xi \in \partial_\infty \ucM$ and $x \in \ucM$.
	Then, for any $v,w \in T_x\ucM$ unitary vectors orthogonal to each other, it holds
	\[ (db_{\xi,x_0}(w))^2 \le C' \cdot \mathrm{Hess}\ b_{\xi,x_0}(v,v) \]
	for some $C' > 0$ independent of $x,v,w$ and $\xi$.
\end{prop}

Let us prove that $K_a$ has at most one ``small'' eigenvalue.
Let $v_1$ and $v_2$ 
 be two orthogonal unitary eigenvectors of $K_a$.
 For every $\xi \in \partial_\infty \ucM$, denote by $v_1^\xi$ and $v_2^\xi$ the projections of $v_1$ and $v_2$ on the codimension-$1$ subspace of $T_{s(a)}$ corresponding to the horosphere centered at $\xi$.
Since $v_1$ and $v_2$ are orthogonal, it must hold either $\lVert v_1^\xi \rVert^2 \ge 1/2$ or $\lVert v_2^\xi \rVert^2 \ge 1/2$.
It follows from \Cref{prop:Hess} that either $\mathrm{Hess}\ b_{\xi,x_0}(v_1,v_1) \ge C/2$ or $\mathrm{Hess}\ b_{\xi,x_0}(v_2,v_2) \ge C/2$.
In any case, we have
$\mathrm{Hess}\ b_{\xi,x_0}(v_1,v_1) + \mathrm{Hess}\ b_{\xi,x_0}(v_2,v_2) \ge C/2$.
After integrating over $\xi \in \partial_\infty \ucM$, we obtain that $K_a(v_1,v_1) + K_a(v_2,v_2) \ge \frac{C}{2} \sum_{i=1}^{k+1}a_i^2 \ge \frac{C}{2(k+1)}$.
Therefore, at least one of $K_a(v_1,v_1)$ and $K_a(v_2,v_2)$ must be $\ge \frac{C}{4(k+1)}$.

Hence, all eigenvalues of $K_a$ are $\ge \frac{C}{4(k+1)}$, except possibly the smallest one, which we call $\lambda$. Let $v$ be a  unit eigenvector of $K_a$ relative to $\lambda$.
We now proceed to show that $d-1$ eigenvalues of $H_a$ are $\le C'\lambda$, where $C'$ is the constant appearing in \Cref{prop:compare_HK}.

Let $W \subset ds(V)$ be the codimension-$1$ subspace orthogonal to $v$.
By integrating over $\partial_\infty \ucM$ the inequality of \Cref{prop:compare_HK}, we obtain that $H(w,w) \le C'\lambda\lVert w\rVert^2$ for every $w \in W$.
Since $\dim(W) = d-1$, this readily implies that the $d-1$ smallest eigenvalues of $H_a$, i.e., all the eigenvalues except possibly one, are $\le C'\lambda$, as claimed.

Plugging this information into \eqref{eq:J_estimate}, we obtain a uniform upper bound for $\lvert \det(ds_{\restriction V}) \rvert$, which shows that the barycentric straightening has bounded Jacobians in degrees $\ge 3$. This concludes the proof of Theorem~\ref{bounded:jac:thm}.

\subsection{Unbounded 2-Jacobians in hyperbolic triangles.}\label{unbounded:sub}

We have just seen that the barycentric straightening has bounded Jacobians in degrees $\ge 3$.
In this subsection we construct examples which show that it has not bounded Jacobians in degree $2$, in general.
Indeed, already for triangles in the hyperbolic plane $\mathbb{H}^2$, the parametrization given by the barycenter method has arbitrarily large Jacobians.

Consider the Poincar\'e disc model $\mathbb{D} = \{x \in \mathbb{C} :  \norm{x} < 1 \}$ of the hyperbolic plane.
Then, for any $\xi \in \partial\mathbb{D}$, the Busemann function $b_{\xi}$ has gradient
\[ \mathrm{grad}\ b_{\xi}(x) = \frac{1-\norm{x}^2}{2} \cdot x + \frac{(1-\norm{x}^2)^2}{2} \cdot \frac{x-\xi}{ \norm{x-\xi}^2}. \]
Being invariant under isometries fixing the point $0 \in \mathbb{D}$, the Patterson-Sullivan measure $\nu_0$ is simply the normalized Lebesgue measure on $\partial \mathbb{D}= S^1$.
A standard calculation gives then
\[ \mathrm{grad}\ \mathcal{B}_{\nu_0}(x) = \int_{\partial\mathbb{D}} \mathrm{grad}\ b_{\xi}(x)\ d\nu_0(\xi) = \frac{1-\norm{x}^2}{2} \cdot x.\]
The latter is a vector in $T_x\mathbb{D}$ whose (hyperbolic) norm is $\norm{x} = \tanh(d_x/2)$, where $d_x$ is the hyperbolic distance between $x$ and $0$, and pointing away from the origin $0$.
Differentiating further, we obtain that $\nabla\mathrm{grad}\ \mathcal{B}_{\nu_0}(x)$, which is an endomorphism of $T_x\mathbb{D}$, is represented, with respect to the canonical basis of $\R^2 = T_x\mathbb{D}$, by the matrix
\[\begin{pmatrix}
	\frac{1-\norm{x}^2}{2} + x_2^2 & -x_1x_2 \\
	-x_1x_2 & \frac{1-\norm{x}^2}{2} + x_1^2
\end{pmatrix}.\]
As expected from the symmetries of the configuration, this endomorphism has eigenspaces in the ``radial'' and ``angular'' directions; the respective eigenvalues are:
\begin{itemize}
	\item $\frac{1 - \norm{x}^2}{2} = \frac{1 - \tanh^2(d_x/2)}{2}$ in the radial direction;
	\item $\frac{1 + \norm{x}^2}{2} = \frac{1 + \tanh^2(d_x/2)}{2}$ in the angular direction.
\end{itemize}
Note that in the limit $\norm{x} \to 1$, i.e., when $x$ is approaching a point at infinity $\xi \in \partial \mathbb{D}$, the endomorphisms converge to the orthogonal projection on $\langle \xi \rangle^\perp$.
 
Now that we have gained an understanding of $\mathcal{B}_{\nu_0}$, its gradient, and its Hessian (and also of the analogous quantities relative to Patterson-Sullivan measures $\nu_x$ based at any point $x$, by equivariance under isometries), we can use \eqref{eq:derived_relation} to compute the Jacobian of a straightened triangle.

\begin{figure}[ht]
	\centering
	\includegraphics[trim={2.5cm 2cm 2cm 2.5cm}, clip, width=.45\textwidth]{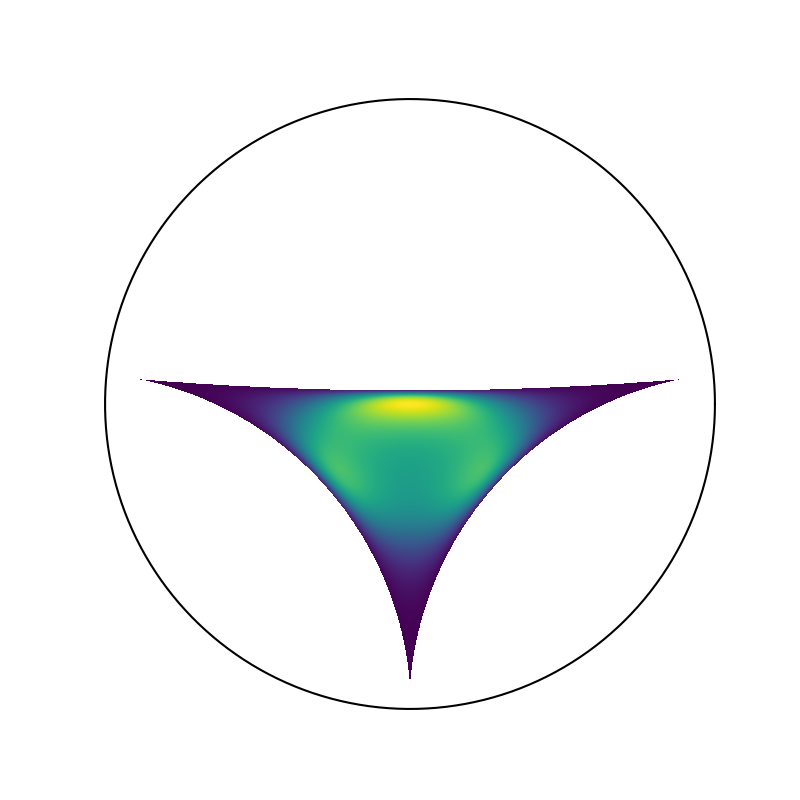}
	\hfill
	\includegraphics[trim={2.5cm 2cm 2cm 2.5cm}, clip, width=.45\textwidth]{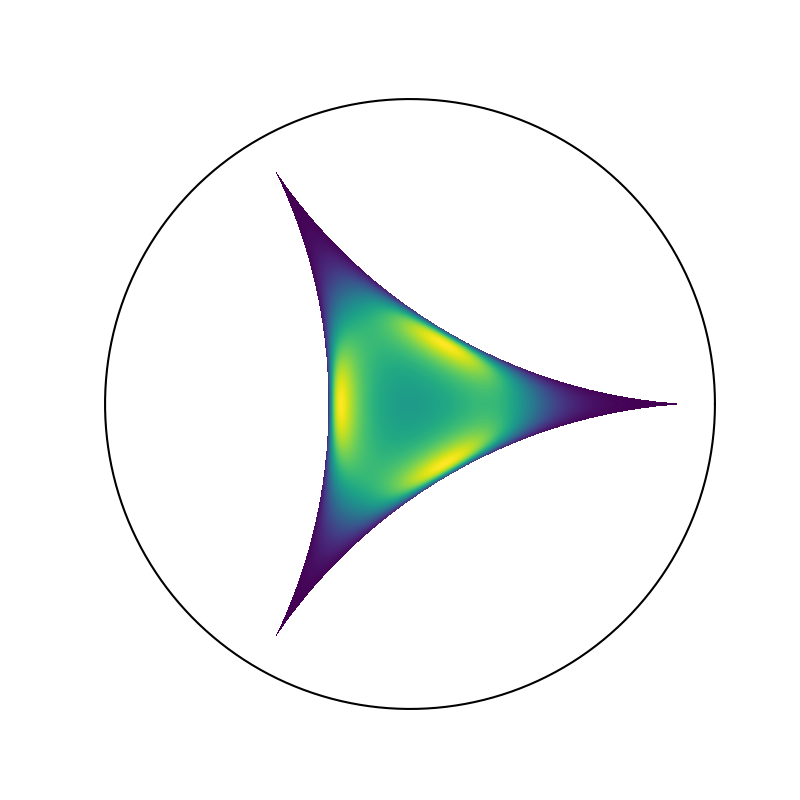}
	\caption{Two hyperbolic triangles in the Poincar\'e disk model. The color represents the value of the Jacobian of the barycentric parametrization: lighter colors correspond to a larger Jacobian (i.e., areas in which the parametrizing triangle is ``stretched'' more).}
	\label{fig:triangles}
\end{figure}

In particular, we pick one specific triangle obtained in this way (see the triangle on the left in \Cref{fig:triangles}): fix a (small) angle $\alpha \in (0,\pi/2)$; consider the geodesic rays that start from $0$ and end at $\xi_1 = e^{\iota\alpha}$, $\xi_2 = e^{\iota(\pi-\alpha)}$ and $\xi_3 = -\iota \in \partial \mathbb{D}$; fix then a (big) distance $D$, and take points $x_1,x_2,x_3$ on the three geodesic rays at distance $D$ from $0$.
Let $s\colon\Delta^2\to\mathbb{D}$ be the straight simplex with vertices $x_1,x_2,x_3$ obtained with the barycenter method. Recall that, for  hyperbolic spaces, the barycentric and the geodesic straightenings coincide up to reparametrization, hence the image of $s$ is just the convex hull of $x_1,x_2,x_3$; in particular, the origin
 $0 \in \mathbb{D}$ of the Poincar\'e model of $\mathbb{H}^2$ lies in the image of $s$ (close to the midpoint of one of the edges). We will compute the Jacobian of $s$
 at $a=s^{-1}(0)\in\Delta^{2}$. By definition, if $a = (a_1,a_2,a_3)$, then
 $$
 0 = s(a) =  \mathrm{Bar}(a_1^2\nu_{x_1} + a_2^2\nu_{x_2} + a_3^2\nu_{x_3})\ .
 $$
For convenience, define $T = \tanh(D/2)$, which is the Euclidean distance of $x_1,x_2,x_3$ from $0$.
From the first-order differential equation \eqref{eq:stationary} for $s$, which can be rewritten as
\[ a_1^2\ \mathrm{grad}\;\mathcal{B}_{\nu_{x_1}}(s(a)) + a_2^2\ \mathrm{grad}\;\mathcal{B}_{\nu_{x_2}}(s(a)) + a_3^2\ \mathrm{grad}\;\mathcal{B}_{\nu_{x_3}}(s(a)) = 0, \]
and  from the fact  that the vectors $\mathrm{grad}\ \mathcal{B}_{\nu_{x_i}}(0) \in T_0\mathbb{D}$ point towards $-\xi_i$, all with the same norm $T$, it follows that 
\begin{align*}a_1 = a_2 = \frac{1}{2 + \sqrt{2 \sin\alpha}}\ , & & a_3 = \frac{\sqrt{2 \sin\alpha}}{2 + \sqrt{2 \sin\alpha}} \ .\end{align*}

We also deduce from \eqref{eq:derived_relation} that the differential $ds\colon T_a\Delta^2 \to T_0\mathbb{D}$ satisfies the condition
\begin{equation}\label{eq:stima_ds} \left( \sum_i a_i^2\; \nabla \mathrm{grad}\;\mathcal{B}_{\nu_{x_i}}\right) (ds(u)) = -2 \left(\sum_i a_i u_i\; \mathrm{grad}\;\mathcal{B}_{\nu_{x_i}}\right), \end{equation}
which holds for every $u \in T_a\Delta^2$.
With respect to the canonical basis of $T_0\mathbb{D} = \mathbb{R}^2$, the quantities in the latter formula are matrices and vectors whose coefficients are continuous functions of $\alpha \in (0,\pi/2)$ and $T \in [0,1]$.
For simplicity, we consider the limit case $T = 1$, which corresponds to $D \to \infty$ and simplifies calculations because then each endomorphism $\nabla \mathrm{grad}\;\mathcal{B}_{\nu_{x_i}}$ is just the orthogonal projection onto $\langle \xi_i\rangle^\perp$.
Then, a computation shows that the endomorphism on the left-hand side of \eqref{eq:stima_ds} is represented by the matrix
\[\frac{2}{(2 + \sqrt{2\sin\alpha})^2}\begin{pmatrix}
	\sin^2\alpha & 0 \\
	0 &  \cos^2\alpha + \sin\alpha
\end{pmatrix}.\]

Consider $u = (1,-1,0) \in T_a\Delta^2$.
Then, the right-hand side of \eqref{eq:stima_ds} becomes (again, in the limit $T = 1$) equal to
$\frac{4 \cos\alpha}{2 + \sqrt{2\sin\alpha}} e_1$, which implies
\[ ds\begin{pmatrix}1\\-1\\0\end{pmatrix} = \frac{2 \cos\alpha\ (2 + \sqrt{2\sin\alpha})}{\sin^2\alpha} e_1 \in T_0\mathbb{D}\ .\]
Similarly, considering $u = (-1,-1,2)$, the right-hand side of \eqref{eq:stima_ds} evaluates as $\frac{-4(\sin\alpha + \sqrt{2\sin\alpha})}{2 + \sqrt{2\sin\alpha}} e_2$, which gives
\[ ds\begin{pmatrix}-1\\-1\\2\end{pmatrix} = \frac{-2 (\sin\alpha+\sqrt{2\sin\alpha})(2+\sqrt{2\sin\alpha})}{\cos^2\alpha + \sin\alpha} e_2 \in T_0\mathbb{D}\ .\]
Therefore, the determinant of $ds$ at the point $a$, with respect to the basis of $T_a\Delta^2$ given by the two vectors $u$ chosen above, and the canonical basis of $T_0\mathbb{D}=\mathbb{R}^2$, is
\[ \frac{-4 \cos\alpha\ (2 + \sqrt{2\sin\alpha})^2 }{\cos^2\alpha + \sin\alpha} \cdot \frac{\sin\alpha + \sqrt{2\sin\alpha}}{\sin^2\alpha}, \]
which diverges as $\alpha$ approaches $0$.

This means that, by choosing $\alpha$ sufficiently small and $D$ sufficiently large, we obtain straightened triangles in $\mathbb{H}^2$ with arbitrarily large Jacobians at certain points.

 \section{Cup product of De Rham classes}

  This section is entirely devoted to  the proof of Theorem~\ref{main:thm}. In fact, we will prove a slightly more general result, which can be applied to the barycentric straightening, but also possibly to other straightenings satisfying the specific key properties needed in the proof of the theorem.

  The first property we need is the boundedness of Jacobians in sufficiently high degrees, which we have established for the barycentric straightening in Subsection~\ref{bounded:jac:sub}.
  The second property is a compatibility condition between the straightening and the cross product, as we discuss below in Subsection~\ref{cross:sec} (for the case of the barycentric straightening, the compatibility condition is proved in \Cref{prop:bdd_hom}).
  Then, we prove in Subsection~\ref{proof:sec} the following result.
  
   \begin{thm}\label{general:thm}
   Let $M$ be a closed negatively curved manifold.
   Let $\str_\bullet$ be a straightening on $M$ having bounded Jacobians in degree
   $\geq d$, and satisfying the conclusion of \Cref{prop:bdd_hom}.
   For every $k\geq d$, let $\Psi_{\STR}^k\colon C\Omega^k(M)\to H^k_b(M)$ be the map described in Definition~\ref{def:psi}. Then,
   $$
\Psi_{\STR}^{p+q}(\omega_1\wedge\omega_2)=\Psi_{\STR}^p(\omega_1)\cup \Psi_{\STR}^q(\omega_2)
$$
for every $p\geq d$, $q\geq d$, and every $\omega_1\in C\Omega^p(M)$,  $\omega_2\in C\Omega^q(M)$.
   \end{thm}
   
   Henceforth, we denote by $M$ a  closed negatively curved manifold of dimension $n$.
      We denote by $\pi_1,\pi_2\colon M\times M\to M$ the projections onto the first and  the second factor, respectively. 

Recall that there exists a \emph{cross product} (also called \emph{Eilenberg-Zilber map}, since it was first introduced in~\cite{EZ})
$$
\times \colon C_\bullet(M)\otimes C_\bullet(M)\to C_\bullet(M\times M)\, ,\quad c_1\otimes c_2\mapsto c_1\times c_2,
$$ 
which is a chain map with respect to the standard boundary operator on $C_\bullet(M)\otimes C_\bullet(M)$ -- see, e.g., \cite[Chapter VI, Section 1]{Bredon}.

The cross product is defined (see, e.g., \cite[Chapter IV, Section  16]{Bredon}) by fixing, for every $p,q \in \mathbb{N}$, a suitable chain $K_{p,q} \in C_{p+q}(\Delta^p\times\Delta^q)$, representing the fundamental cycle of $\Delta^p\times\Delta^q$ relative to its boundary $(\partial \Delta^{p}\times \Delta^{q})\cup (\Delta^{p}\times\partial \Delta^{q})$, so that for any $\sigma_1 \in \Simpl_p(M)$ and $\sigma_2 \in \Simpl_q(M)$ the cross product $\sigma_1 \times \sigma_2 \in C_{p+q}(M\times M)$ is the push forward of $K_{p,q}$ via the map $(x,y) \in \Delta^p\times\Delta^q \mapsto (\sigma_1(x),\sigma_2(y)) \in M\times M$.

A standard choice for $K_{p,q}$ (which we fix from now on) is the so called Eilenberg-Zilber chain, described in \cite{EM}, which is obtained by subdividing the prism $\Delta^p\times\Delta^q$ into affine simplices in a suitable explicit way.
A useful property of this specific triangulation is that it doesn't introduce new vertices, i.e., its vertices are all given by pairs of vertices of $\Delta^p$ and $\Delta^q$. Moreover,
with this choice for $K_{p,q}$, the cross product of two simplices of class $C^1$ is a linear combination of simplices of class $C^1$. 	
	By construction, the cross product is bounded (with respect to the obvious $\ell^1$-norms on $C_\bullet(M)\otimes C_\bullet(M)$ and on $C_\bullet(M\times M)$).

For later purposes we point out the following:

\begin{lemma}\label{dimsbagliate}
Let $\omega_1\in C\Omega^p(M)$, $\omega_2\in C\Omega^q(M)$, let $p'+q'=p+q$ and let
 $\sigma_1\colon \Delta^{p'}\to M$, $\sigma_2\colon \Delta^{q'}\to M$ be  simplices of class  $C^1$. If $p'=p$ and $q'=q$, then 
 $$
 \int_{\sigma_1 \times \sigma_2} \pi_1^*(\omega_1)\wedge \pi_2^*(\omega_2)=\left( \int_{\sigma_1} \omega_1\right)\cdot \left( \int_{\sigma_2} \omega_2\right)\ .
 $$
 Otherwise,
 $$
 \int_{\sigma_1 \times \sigma_2} \pi_1^*(\omega_1)\wedge \pi_2^*(\omega_2)=0\ .
 $$
 
 \end{lemma}
\begin{proof}
It readily follows from the definition of cross product that, if $k\colon \Delta^{p'}\times \Delta^{q'}\to M\times M$ is given by
$k(x,y)=(\sigma_1(x),\sigma_2(y))$, then 
$$
 \int_{\sigma_1\times \sigma_2} \pi_1^*(\omega_1)\wedge \pi_2^*(\omega_2)=\int_{\Delta^{p'}\times \Delta^{q'}} k^*(\pi_1^*(\omega_1)\wedge \pi_2^*(\omega_2)).
$$
Let $p_1\colon \Delta^{p'}\times \Delta^{q'}\to \Delta^{p'}$, $p_2\colon \Delta^{p'}\times \Delta^{q'}\to \Delta^{q'}$ be the projections. Then
$\pi_i\circ k=\sigma_i\circ p_i$ for $i=1,2$, hence
$$
\int_{\Delta^{p'}\times \Delta^{q'}} k^*(\pi_1^*(\omega_1)\wedge \pi_2^*(\omega_2))
 = \int_{\Delta^{p'}\times \Delta^{q'}} p_1^*(\sigma_1^*(\omega_1))\wedge p_2^*(\sigma_2^*(\omega_2))\ .
$$
Suppose now that $p'\neq p$. Then, either $p'<p$ or $q'<q$. In the first case we have $p_1^*(\sigma_1^*(\omega_1))=0$, while in the second we have
$p_2^*(\sigma_2^*(\omega_2))=0$. In any case, $p_1^*(\sigma_1^*(\omega_1))\wedge p_2^*(\sigma_2^*(\omega_2))=0$ and
$$
 \int_{\sigma_1\times \sigma_2} \pi_1^*(\omega_1)\wedge \pi_2^*(\omega_2)= \int_{\Delta^{p'}\times \Delta^{q'}} p_1^*(\sigma_1^*(\omega_1))\wedge p_2^*(\sigma_2^*(\omega_2))=0\ .
$$
This proves the lemma when $p'\neq p$ (hence, $q'\neq q$). On the other hand, If $p=p'$ and $q=q'$, 
Fubini's Theorem implies that
\begin{align*}
 \int_{\sigma_1\times \sigma_2} \pi_1^*(\omega_1)\wedge \pi_2^*(\omega_2) &= \int_{\Delta^{p}\times \Delta^{q}} p_1^*(\sigma_1^*(\omega_1))\wedge p_2^*(\sigma_2^*(\omega_2))\\ &=\left( \int_{\sigma_1} \omega_1\right)\cdot \left( \int_{\sigma_2} \omega_2\right)\ .\qedhere
\end{align*}
\end{proof}

Henceforth, we will denote by $P_\bullet  \colon C_\bullet(M)\otimes C_\bullet(M)\to C_\bullet(M\times M)$ the cross product.
The cross product admits a chain homotopy inverse 
$$
\theta_\bullet \colon C_\bullet(M\times M)\to C_\bullet(M)\otimes C_\bullet(M)
$$
such that
$$
P_k\circ \theta_k  - \mathrm{Id}_k=\partial_{k+1}\circ H_k \pm H_{k-1}\circ \partial_k\ ,
$$
where $H_\bullet \colon C_\bullet(M\times M)\to C_{\bullet +1}(M\times M)$ is a chain homotopy. Being constructed via acyclic models, 
both $\theta_\bullet$ and 
the chain homotopy $H_\bullet$ are bounded with respect to the $\ell^1$-norm. By dualizing the equality above (where we denote by $H^{\bullet +1}$ the transpose of $H_\bullet$, since
$H_\bullet$ is homogeneous of degree $+1$), we thus get the equality
\begin{equation}\label{dual:homotopy}
\theta^k\circ P^k  - \mathrm{Id}^k=H^{k+1}\circ \delta^k \pm \delta^{k-1}\circ H^{k}\ ,
\end{equation}
in which all the maps involved preserve the boundedness of cochains.

In fact, we will make use of a specific $\theta_\bullet$, the so-called \emph{Alexander-Whitney map}, which is defined as follows. 
For any singular $k$-simplex $\sigma\colon \Delta^k\to X$ and any $0\leq p\leq k$, we define the \emph{first} $p$-face of $\sigma$ by setting
$\sigma|_{p}=\sigma\circ i|_p\colon \Delta^p\to X$, where $i|_p\colon \Delta^p\to \Delta^k$ is the affine embedding sending the vertices of $\Delta^p$
onto the first $p+1$ vertices of $\Delta^k$, preserving their order. In the same way, we define the \emph{last} 
$p$-face of $\sigma$ by setting
${}_{p} |\sigma=\sigma\circ {}_p | i\colon \Delta^p\to X$, where $ {}_p | i\colon \Delta^p\to \Delta^k$ is the affine embedding sending the vertices of $\Delta^p$
onto the last $p+1$ vertices of $\Delta_k$, preserving their order.

By definition of product topology,
any singular $k$-simplex $\sigma\colon \Delta^k\to M\times M$ is of the form $\sigma(x)=(\sigma_1(x), \sigma_2(x))$, where $\sigma_1,\sigma_2\colon \Delta^k \to M$
are singular simplices. The Alexander-Whitney map is then the linear extension $\theta_{k}\colon C^{k}(M\times M)\to \oplus_{p+q=k} C_p(M)\otimes C_q(M)$
of the map 
$$
\sigma=(\sigma_1,\sigma_2) \ \longmapsto \ \sum_{p+q=k} (\sigma_1 |_p)\otimes ({}_q | \sigma_2)\ . 
$$
For later purposes, we also observe that $\str_p(\sigma|_p)=(\str_k(\sigma))|_p$ and $\str_p({}_p |\sigma)={}_p|\str_k(\sigma)$, i.e., the straightening operator commutes
with taking the first and the last $p$-face of a singular simplex.

Let $D\colon M\to M\times M$ be the diagonal map $D(x)=(x,x)$. If $\alpha_1\in C^p_b(M)$, $\alpha_2\in C^q_b(M)$ are 
bounded cocycles, then the cup product between the bounded classes they represent is, by definition, represented by the cochain $\alpha_1 \cup \alpha_2\in C^{p+q}_b(M)$ such that
$$
(\alpha_1 \cup \alpha_2) (c)=(\alpha_1\otimes \alpha_2) (\theta_{p+q}(D_\bullet(c)))\ .
$$
Therefore, if $\sigma\colon \Delta^{p+q}\to M$ is a singular simplex, then 
$$
(\alpha_1 \cup \alpha_2) (\sigma)=\alpha_1 (\sigma|_p) \cdot \alpha_2 ({}_q |\sigma)\ .
$$ 

\subsection{Straightening and cross product}\label{cross:sec}

Henceforth, we fix a straightening $\STR_\bullet$ on $M$.
The straightening on $M$ induces a straightening $\str_\bullet^2$ on $M\times M$: any singular simplex $\sigma\colon \Delta^k\to M\times M$
is equal to $\sigma=(\sigma_1,\sigma_2)$, where $\sigma_i\colon \Delta^k\to M$ is a singular simplex, and we then
set $\str^2_k(\sigma)=(\str_k(\sigma_1),\str_k(\sigma_2))$. 
Observe that for every singular $k$-simplex $\sigma$ with values in $M$ we have
$\str^2_k(D(\sigma))=D(\str_k(\sigma))$.

The straightening also induces a chain map
\[ \str_\bullet^\otimes \colon C_\bullet(M) \otimes C_\bullet(M) \to C_\bullet(M) \otimes C_\bullet(M) \]
sending any $\sigma_1\otimes\sigma_2$ to $\str(\sigma_1) \otimes \str(\sigma_2)$.
Consider the diagram of chain maps
\begin{equation}\label{eq:cross_str}\begin{tikzcd}
	C_\bullet(M) \otimes C_\bullet(M) \ar[r,"\str_\bullet^\otimes"]\ar[d,"P_\bullet"] & C_\bullet(M) \otimes C_\bullet(M) \ar[d,"P_\bullet"] \\
	C_\bullet(M\times M) \ar[r,"\str_\bullet^2"] & C_\bullet(M\times M)
\end{tikzcd}\end{equation}
where $P_\bullet$ is the cross product, which we assume is defined using the Eilenberg-Zilber chains.
In order to prove \Cref{general:thm}, we need that the diagram \eqref{eq:cross_str} commutes \emph{up to a homotopy with bounded volumes}.
\begin{defn}\label{def:volume}
	Let $\sum a_i \tau_i \in C_k(M\times M)$ be a chain, where $a_i \in \mathbb{R}$ and $\tau_i$ are distinct singular simplices of class  $C^1$, or, more generally, having enough regularity (e.g., Lipschitz maps) so that their $k$-dimensional volume is defined and finite.
	The \emph{volume} of such a chain is the sum $\sum \lvert a_i \rvert \vol_{k}(\tau_i)$.
\end{defn}
\begin{defn}\label{def:hom_bddvol}
	We say that a homotopy between the two chain maps $\str^2_\bullet\circ P_\bullet$ and $P_\bullet \circ \str^\otimes_\bullet$, given by linear maps
	\[F_k \colon \bigoplus_{p+q=k} C_p(M) \otimes C_q(M) \to C_{k+1}(M\times M), \]
	has bounded volumes if,
	for every $p,q \in \mathbb{N}$ there is a bound on the volumes of $F_{p+q}(\sigma_1\otimes\sigma_2)$, for $\sigma_1 \in \Simpl_p(M)$ and $\sigma_2 \in \Simpl_q(M)$, independent of $\sigma_1$ and $\sigma_2$.
	Here, we are assuming that the chains $F_{p+q}(\sigma_1\otimes\sigma_2)$ are regular enough as in \Cref{def:volume}.	
\end{defn}

\begin{prop}\label{prop:bdd_hom}
	Assume that $\str_\bullet$ is the barycentric straightening on $M$.
	Then, the diagram \eqref{eq:cross_str} commutes up to a homotopy with bounded volumes.
\end{prop}
\begin{proof}
	We actually prove that there is a homotopy with volumes all equal to $0$. Let $p,q \in \mathbb{N}$ and fix $\sigma_1 \in \Simpl_p(M)$, $\sigma_2 \in \Simpl_q(M)$.
	Let $K_{p,q}$ be the triangulation of $\Delta^p\times\Delta^q$ defining the Eilenberg-Zilber chain.
	
	Consider the following maps $G_0, G_1 \colon \Delta^p\times\Delta^q \to M\times M$.
	\begin{itemize}
		\item $G_0(x,y) = (\str_p(\sigma_1)(x), \str_q(\sigma_2)(y))$;
		\item For every $(p+q)$-simplex $\tau\colon\Delta^{p+q}\to\Delta^p\times\Delta^q$ appearing in the chain $K_{p,q}$ (recall that such a simplex is an affine embedding), let $G_\tau:\Delta^{p+q}\to M\times M$ be given by the composition of $\tau$ with the map $(x,y)\mapsto (\sigma_1(x),\sigma_2(y))$.
		Then, $G_1\colon\Delta^p\times\Delta^q\to M\times M$ is obtained by glueing back together the straightened simplices $\str^2_{p+q}(G_\tau)$.
		That is, if $w \in \Delta^p\times\Delta^q$ is in the image of $\tau$, then $G_1(w) = \str^2_{p+q}(G_\tau)(\tau^{-1}(w))$.
		Notice that this definition is well posed because of the compatibility of the straightening $\str_\bullet^2$ with the restriction to faces.
	\end{itemize}
	By construction, the chains \[P_{p+q}\str^\otimes_{p+q}(\sigma_1\otimes \sigma_2) = \str_p(\sigma_1) \times \str_q(\sigma_2)\]
	and
	\[\str^2_{p+q}P_{p+q}(\sigma_1\otimes\sigma_2) = \str^2_{p+q}(\sigma_1\times\sigma_2)\]
	are equal to the push-forwards of $K_{p,q}$ via the maps $G_0$ and $G_1$, respectively.
	To prove the proposition, we construct a homotopy $G\colon\Delta^p\times\Delta^q\times[0,1] \to M\times M$ between $G_0$ and $G_1$; then, the chain homotopy is obtained form $G$ by triangulating $K_{p,q} \times [0,1]$ in a standard way (e.g., by using again the Eilenberg-Zilber method on each sub-prism $\Delta^{p+q}\times [0,1]$).
	
	Let $w \in \Delta^p\times\Delta^q$, and suppose it is in the image of a certain $\tau$ as above.
	Thus, $w = \tau(a_1, \dots, a_{p+q+1})$ for some $(a_1,\dots,a_{p+q+1}) \in \Delta^{p+q}$.
	Write $\tau(e_i) = (e_{j^1_i}, e_{j^2_i})$ for $i \in \{1,\dots,p+q+1\}$.
	Then, by definition,
	$G_1(w)$ is the projection on $M \times M$ of the point
	\begin{equation}\label{eq:sub_bar} \left(\mathrm{Bar}\left(\sum_{i=1}^{p+q+1} a_i^2 \nu_{\widetilde\sigma_1(e_{j^1_i})}\right), \mathrm{Bar}\left(\sum_{i=1}^{p+q+1} a_i^2 \nu_{\widetilde\sigma_2(e_{j^2_i})}\right) \right) \in \widetilde{M}\times\widetilde{M},\end{equation}
	where $\widetilde\sigma_1$ and $\widetilde\sigma_2$ are lifts of the singular simplices $\sigma_1$ and $\sigma_2$ to the universal cover $\widetilde{M}$.
	Consider the first entry of the pair in \eqref{eq:sub_bar}. By grouping together the indices $i$ for which the $j_i^1$ are equal (and re-normalizing the coefficients), we can write it as
	\[ \mathrm{Bar}\left(\sum_{j=1}^{p+1} b_j^2 \nu_{\widetilde\sigma_1(e_j)}\right), \]
	for a certain $(b_1,\dots,b_{p+1}) \in \Delta^p$, which depends only on $(a_1,\dots,a_{p+q+1})$ and on the vertices of $\tau$.
	We can do the same for the second entry; thus, we obtain a map $\Delta^{p+q} \to \Delta^p \times \Delta^q$ which is smooth on the interior of every face of $\Delta^{p+q}$, and is Lipschitz.
	By patching together these maps on the various simplices of $K_{p,q}$ (we can do this because they are compatible on the shared faces), we obtain a map $r_{p,q}\colon\Delta^p\times\Delta^q \to \Delta^p\times\Delta^q$ such that $G_1 = G_0 \circ r_{p,q}$.
	The map $r_{p,q}$ is Lipschitz, and its restriction to the interior of each simplex of $K_{p,q}$ is smooth. Notice also that:
	\begin{itemize}
		\item The map $r_{p,q}$ depends only on $p,q \in \mathbb{N}$;
		\item The construction is compatible with the restriction to faces, i.e., the restrictions of $r_{p,q}$ to the prisms composing $\partial(\Delta^p\times\Delta^q)$, which are canonically identified with $\Delta^{p-1}\times\Delta^q$ or $\Delta^p\times \Delta^{q-1}$, coincide (up to these identifications) with $r_{p-1,q}$ and $r_{p,q-1}$.
	\end{itemize}
	We can now define the homotopy $G$ by setting
	\[ G(x,y,t) = G_0((1-t)\cdot(x,y) + t\cdot r_{p,q}(x,y)).\]
	It is clear that $G(x,y,0) = G_0(x,y)$ and $G(x,y,1) = G_1(x,y)$ for every $(x,y) \in \Delta^p\times\Delta^q$.
	Moreover, since both the family of maps $r_{p,q}$ and the straightening (used to define $G_0$) are compatible with the restrictions to faces, the same is true for the construction of $G$ (as a map depending on $\sigma_1$ and $\sigma_2$).
	Thus, by taking the push-forward via $G$ of a standard triangulation of $K_{p,q}\times [0,1]$, as already mentioned, we obtain a chain homotopy.
	
	To conclude the proof, it is enough to observe that the constructed chain homotopy has volume $0$, because the map $G\colon\Delta^p\times\Delta^q\times[0,1] \to M\times M$, by construction, factors through $\Delta^p\times\Delta^q$, which has dimension strictly smaller than $p+q+1$.
\end{proof}
\begin{rmk}
	The geodesic straightening also satisfies the conclusion of \Cref{prop:bdd_hom}; actually, the situation is much simpler in this case, because the diagram \eqref{eq:cross_str} commutes on the nose.
\end{rmk}

\subsection{Proof of Theorem~\ref{general:thm}}\label{proof:sec}

Henceforth, we assume that the straightening $\str_\bullet$ has bounded Jacobians in degrees
$\geq d$, and that the diagram \eqref{eq:cross_str} commutes up to a homotopy with bounded volumes.
Therefore, we have a chain homotopy $F_\bullet \colon C_\bullet(M)\otimes C_\bullet(M) \to C_\bullet(M\times M)$ as in \Cref{def:hom_bddvol}.
For ease of notation, we also denote simply by $\Psi^k$ the map $\Psi^k_\str$, $k\geq d$.

Let us now fix $\omega_1\in C\Omega^p(M)$, $\omega_2\in C\Omega^q(M)$ with $p\geq d$, $q\geq d$.
We define a cocycle $$I(\omega_1,\omega_2)\in C^{p+q}(M\times M)$$ by setting
$$
I(\omega_1,\omega_2)(\sigma)=\int_{\str^2_{p+q}(\sigma)} \pi_1^*(\omega_1)\wedge \pi_2^*(\omega_2)\ 
$$
for every singular simplex $\sigma\colon \Delta^{p+q}\to M\times M$.
The fact that $I(\omega_1,\omega_2)$ is a cocycle readily follows
from the fact that $\omega_1$ and $\omega_2$ are closed, hence $\pi_1^*(\omega_1)\wedge \pi_2^*(\omega_2)$ is also so.

The fact that the straightening we are considering has bounded Jacobians in degrees $\geq d$ plays a fundamental role in the following result, which represents a key step
in the proof of Theorem~\ref{general:thm}.

\begin{prop}\label{Hbounded}
Let $\omega_1\in C\Omega^p(M)$, $\omega_2\in C\Omega^q(M)$.
If $p\geq d$ and $q\geq d$, then the cocycle $I(\omega_1,\omega_2)$ is bounded.
\end{prop}
\begin{proof}

Let $\sigma = (\sigma'_1,\sigma'_2)\colon \Delta^{p+q}\to M\times M$ be a singular simplex.
If we set $\sigma_1=\str_{p+q}(\sigma'_1)$, $\sigma_2=\str_{p+q}(\sigma_2')$, we then have
$\str^2_{p+q}(\sigma)=(\sigma_1,\sigma_2)$, hence
$$
I(\omega_1,\omega_2)(\sigma)=\int_{(\sigma_1,\sigma_2)} \pi_1^*(\omega_1)\wedge \pi_2^*(\omega_2)=\int_{\Delta^{p+q}} \sigma_1^*(\omega_1)\wedge \sigma_2^*(\omega_2)\ .
$$
Since $M$ is compact, both $\omega_1$ and $\omega_2$ are bounded, i.e., there exist $C_1,C_2\in\mathbb{R}$ such that
$|\omega_1(v_1\wedge\dots\wedge v_p)|\leq C_1 \norm{v_1\wedge\dots\wedge v_p}$ and $|\omega_2(w_1\wedge\dots\wedge w_q)|\leq C_2 \norm{w_1\wedge\dots\wedge w_q}$, whenever
$v_1,\dots,v_p$ and $w_1,\dots,w_q$ are frames at any point of $M$. Moreover, since we are assuming that our straightening has bounded Jacobians in degrees $\geq d$,
there exists $K\in\mathbb{R}$ such that, for every $x\in \Delta^{p+q}$ and every orthonormal frames
$e_1,\dots,e_p$ and $e'_1,\dots,e'_q$ at $x$, we have
$$
\|d(\sigma_1)_x (e_1\wedge\dots\wedge e_p)\|\leq K\, ,\qquad \|d(\sigma_2)_x (e'_1\wedge\dots\wedge e_q')\|\leq K\ .
$$
Putting together these inequalities, for every $x\in \Delta_{p+q}$ and every orthonormal frame
$e_1,\dots, e_{p+q}$ at $x$ we get
$$
\left| \left(\sigma_1^*(\omega_1)\wedge \sigma_2^*(\omega_2)\right)(e_1\wedge\dots\wedge e_{p+q})\right| \leq \frac{(p+q)!}{p!\, q!}\, C_1\, C_2\,  K^2\ .
$$
This shows that the form $\sigma_1^*(\omega_1)\wedge \sigma_2^*(\omega_2)$ is bounded independently of $\sigma_1$ and $\sigma_2$, which implies in turn
that $I(\omega_1,\omega_2)$ is bounded, as desired.
\end{proof}

\begin{prop}\label{Dprop}
We have 
$$
[I(\omega_1,\omega_2)]=[\theta^{p+q}(P^{p+q}(I(\omega_1,\omega_2)))]\qquad \text{in}\ H^{p+q}_b(M\times M)\ .
$$
\end{prop}
\begin{proof}
Let $n=p+q$. By applying~\eqref{dual:homotopy}  to the cochain $I(\omega_1,\omega_2)$, and recalling that 
$\delta^{n}(I(\omega_1,\omega_2))=0$,
we get
$$
\theta^{n}(P^{n}(I(\omega_1,\omega_2))  -I(\omega_1,\omega_2)=\pm \delta^{n-1}(H^{n}(I(\omega_1,\omega_2)))\ .
$$
Since the homotopy $H^\bullet$ preserves the boundedness of cochains, and $I(\omega_1,\omega_2)$ is bounded by \Cref{Hbounded}, this readily implies the conclusion.
\end{proof}

The previous proposition implies in particular that, if $D^\bullet\colon H^{\bullet}_b(M\times M)\to H^{\bullet}_b(M)$ is the map
induced by the diagonal embedding $D\colon M\to M\times M$ in bounded cohomology, then
$$
D^{p+q}([I(\omega_1,\omega_2)])=D^{p+q}([\theta^{p+q}(P^{p+q}(I(\omega_1,\omega_2)))])\qquad \text{in}\ H^{p+q}_b(M)\ .
$$
Therefore, in order to prove that
$$
\Psi^{p+q}(\omega_1\wedge\omega_2)=\Psi^p(\omega_1)\cup \Psi^q(\omega_2)
$$
it will suffice to show that
\begin{equation}\label{seconda}
D^{p+q}([I(\omega_1,\omega_2)])=\Psi^{p+q}(\omega_1\wedge \omega_2)
\end{equation}
and
\begin{equation}\label{terza}
D^{p+q}([\theta^*(P^*(I(\omega_1,\omega_2)))]) =\Psi^{p}(\omega_1)\cup\Psi^q( \omega_2)
\end{equation}
in $H^{p+q}_b(M)$. 

The proof of~\eqref{seconda} is immediate: in fact,
for every singular simplex $\sigma\colon \Delta^{p+q}\to M$ we have
\begin{align*}
I(\omega_1,\omega_2)(D\circ\sigma)& =\int_{\str^2_{p+q}(D\circ\sigma)} \pi_1^*(\omega_1)\wedge \pi_2^*(\omega_2)\\ &=
\int_{D\circ \str_{p+q}(\sigma)} \pi_1^*(\omega_1)\wedge \pi_2^*(\omega_2)\\
&= \int_{\str_{p+q}(\sigma)} D^*( \pi_1^*(\omega_1)\wedge \pi_2^*(\omega_2))=\int_{\str_{p+q}(\sigma)} \omega_1\wedge \omega_2\ ,
\end{align*}
which gives exactly~\eqref{seconda}.

In order to prove~\eqref{terza}, let $\sigma\colon \Delta^{p+q}\to M$ be a singular simplex.
Denote by $D_\bullet \colon C_\bullet(M) \to C_\bullet(M\times M)$ the chain map induced by $D \colon M \to M\times M$.

We first point out that Lemma~\ref{dimsbagliate}  implies that
\begin{align}\label{eq:int_formula}\begin{split}
\sum_{p'+q'=p+q} &\int_{\str_{p'}(\sigma|_{p'})\times\str_{q'}( _{q'}\!|\sigma)} \pi_1^*(\omega_1)\wedge \pi_2^*(\omega_2)
\\ &=\left(\int_{\str_{p}(\sigma|_{p})} \omega_1\right) \cdot \left(\int_{\str_{q( _{q}\!|\sigma)}} \omega_2\right)\ .
\end{split}\end{align}
Now, we compute
\begin{align}\label{eq:compute}
D^{p+q}([\theta^*(P^*(I(\omega_1,\omega_2)))])(\sigma) = \int_{\str^2_{p+q}P_{p+q}\theta_{p+q}D_{p+q}\sigma} \pi_1^*(\omega_1)\wedge \pi_2^*(\omega_2).
\end{align}
The chain over which we are integrating the form $\pi_1^*(\omega_1)\wedge \pi_2^*(\omega_2)$ can now be written as
\begin{align*}
	\str^2_{p+q}P_{p+q}\theta_{p+q}D_{p+q}\sigma
	\ =\  &P_{p+q}\str^\otimes_{p+q}\theta_{p+q}D_{p+q}\sigma \\&+
	F_{p+q-1}\partial \theta_{p+q}D_{p+q}\sigma \\&+ \partial F_{p+q}\theta_{p+q}D_{p+q}\sigma.
\end{align*}

The last summand can be ignored, since it is a boundary and we are integrating a closed differential form on it;
the second summand can be rewritten as $F_{p+q-1}\theta_{p+q-1}D_{p+q-1}\partial\sigma$;
the first summand is equal to
\[ \sum_{p'+q'=p+q} \str_{p'}(\sigma|_{p'})\times\str_{q'}( _{q'}\!|\sigma). \]
Therefore, using \eqref{eq:int_formula}, the computation started in \eqref{eq:compute} continues as follows:
\begin{align*}
	\int_{\str^2_{p+q}P_{p+q}\theta_{p+q}D_{p+q}\sigma} \pi_1^*(\omega_1)\wedge \pi_2^*(\omega_2)
	= \left(\int_{\str_{p}(\sigma|_{p})} \omega_1\right) \cdot \left(\int_{\str_{q( _{q}\!|\sigma)}} \omega_2\right) \\
	+ \int_{F_{p+q-1}\theta_{p+q-1}D_{p+q-1}\partial\sigma} \pi_1^*(\omega_1)\wedge \pi_2^*(\omega_2).
\end{align*}

The last summand is the coboundary of a bounded cochain (the boundedness comes from the fact that $F$ has bounded volumes). 
This proves~\eqref{terza}, and concludes the proof Theorem~\ref{general:thm}, whence of Theorem~\ref{main:thm}.

\bibliography{math_papers}
\bibliographystyle{fram_alpha}
\end{document}